\documentclass{amsart}
\usepackage{amsfonts}
\usepackage{amsmath,amscd}
\usepackage{amssymb}
\usepackage{amsthm}
\usepackage{newlfont}

\long\def\alert#1{\parindent2em\smallskip\hbox to\hsize%?
{\hskip\parindent\vrule%?
\vbox{\advance\hsize-2\parindent\hrule\smallskip\parindent.4\parindent%?
\narrower\noindent#1\smallskip\hrule}\vrule\hfill}\smallskip\parindent0pt}
 \newtheorem{thm}{Theorem}[section]
\newtheorem{cor}[thm]{Corollary}
 \newtheorem{lem}[thm]{Lemma}
 \newtheorem{prop}[thm]{Proposition}
\theoremstyle{definition}
 \newtheorem{defn}[thm]{Definition}
\theoremstyle{remark}

 \numberwithin{equation}{section}
\begin{document}

\title[$c$-nilpotent multiplier of finite  $p$-groups] {$c$-nilpotent multiplier of finite  $p$-groups}
\author[F. Johari]{Farangis Johari}
\author[M. Parvizi]{Mohsen Parvizi}
\author[P. Niroomand]{Peyman Niroomand}

\address{Department of Pure Mathematics\\
Ferdowsi University of Mashhad, Mashhad, Iran}
\email{farangisjohary@yahoo.com}

\address{Department of Pure Mathematics \\
Ferdowsi University of Mashhad, Mashhad, Iran}
\email{parvizi@um.ac.ir}

\address{School of Mathematics and Computer Science \\
Damghan University, Damghan, Iran}
\email{niroomand@du.ac.ir, p$\_$niroomand@yahoo.com}
\thanks{This research was supported by a grant from Ferdowsi University of Mashhad-Graduate Studys ( No. 31659).}


\date{\today}

%\dedicatory{}?

%?
\begin{abstract}
The aim of this work is to find some exact sequences on the $c$-nilpotent multiplier of a group $G$.
We also give an upper bound for the $c$-nilpotent multiplier of finite $p$-groups and give the explicit structure of groups whose take  the upper bound. Finally, we will get the exact structure of the $c$-nilpotent multiplier and determine $c$-capable groups in the class of extra-special and generalized extra-special $p$-groups.
It lets us  to have a vast improvement over the last results on this topic.
\end{abstract}

%%% ----------------------------------------------------------------------
\maketitle
%%% ------
\section{introduction and motivation}
Let $G$ be a  group presented as the quotient $F/R$ of a free group $F$
by a normal subgroup $R.$ The  $c$-nilpotent multiplier of $G$ is defined as
\[\mathcal{M}^{(c)}(G)=\dfrac{\gamma_{c+1}(F)\cap R}{\gamma_{c+1}(R,F)}\]is the Baer invariant of $G$ with respect to the variety of nilpotent groups of class
at most $c$ in which $\gamma_{c+1}(F )$ is the $(c + 1)$-th term of the lower central series of $F$ and $\gamma_1(R, F ) = R, \gamma_{c+1}(R, F ) =
[\gamma_c(R, F ), F ]$, inductively.
The $c$-nilpotent multiplier of $G$ was born in the work of Baer \cite{1}(see also \cite{6, 6A, 11A, 18A} for more details in this topic). It is readily verified that the $c$-nilpotent multiplier $\mathcal{M}^{(c)}(G)$ is abelian and independent of the choice of the free presentation of $G$ (see \cite{green}).
The $1$-nilpotent multiplier of $G$ is denoted by
$\mathcal{M}(G)$ and it is called the Schur multiplier of $G.$ There are many stories involving this
concept and it can be found for instance in
\cite{2, 9, 11, 16, 17, 18, 25, 26, 27}.
The main reason to study the $c$-nilpotent multiplier comes from the
isologism theory of P. Hall \cite{14,15}. It is an instrument for classification of groups into isologism classes. Some results about the Baer invariant can be found in \cite{6, 6A, 11A, 18A, 21, 24, 22, 23}.

The aim of paper is to obtain some more inequalities on the $c$-nilpotent multiplier of a finite $p$-group $G$. For instance, among the results for non-abelian $p$-groups, we give a vast improvement over the result due to Moghaddam \cite{24, 23}. The results also gives an extension of the results obtained recently in \cite{29A}. In the class of extra-special $p$-groups and generalized extra-special $p$-groups, we characterize the explicit structure of the $c$-nilpotent multiplier.
In the other direction, it seems finding a suitable upper bound may be useful to know more about
the $c$-nilpotent multiplier and $c$-capability of groups. It is shown in \cite{22, 23} that for a $p$-group $G$ of order $p^n$,
\[ |\mathcal{M} ^{(c)}(G)| |\gamma_{c+1} (G)|
 \leq |\mathcal{M}^{(c)}(\mathbb{Z}_{p}^{(n)}) |,\] where $\mathbb{Z}_{p}^{(n)}=\underbrace{\mathbb{Z}_{p}\oplus \mathbb{Z}_{p}\oplus \cdots \oplus \mathbb{Z}_{p}}_\text{n-times}.$\\
Using this inequality and Corollary \ref{8} we have $| \mathcal{M} ^{(c)}(G)|  $ is bounded by $ p^{\chi _{c+1}(n-1)},$ where $c\geq 2$. We will show that
the bound is attained exactly when $G$ is elementary abelian similar to the result of \cite[Corollary 2]{2} due to Berkovich. Although we will reduce the upper bound as much as
possible in the case of non-abelian $p$-groups, and as a result we will generalize the work
of Berkovich \cite{2} to the $c$-nilpotent multipliers of finite $p$-groups.
Finally   we are able to identify which of generalized extra-special $p$-groups
are $c$-capable. It also generalizes the result about the capability and $2$-capability of extra-special $p$-groups in \cite{3, 29A}.
\section{Preliminaries}

In this section, we state the concepts and results which will be used in the next sections.\\
{\bf Notation}.
We use techniques involving the concept of basic commutators. Here is the definition.
Let $X$ be an arbitrary subset of a free group, and select an arbitrary total order for $X$. The basic commutators on $X$, their weight $wt$, and the order among them are defined as follows:
\begin{itemize}
\item[(i)]The elements of $X$ are basic commutators of weight one, ordered according to the total order previously chosen.
\item[(ii)]Having defined the basic commutators of weight less than $n$, a basic commutator of weight $n$ is $c_k=[c_i,c_j]$, where:
\begin{itemize}
\item[(a).]$c_i$ and $c_j$ are basic commutators and $wt(c_i)+wt(c_j)=n$, and
\item[(b).]$c_i>c_j$, and if $c_i=[c_s,c_t]$, then $c_j \geq c_t$.
\end{itemize}
\item[(iii)] The basic commutators of weight $n$ follow those of weight less than $n$. The basic commutators of weight $n$ are ordered among themselves in any total order, but the most common used total order is lexicographic order; that is, if $[b_1,a_1]$ and $[b_2,a_2]$ are basic commutators of weight $n$. Then $[b_1,a_1]<[b_2,a_2]$ if and only if  $b_1<b_2$ or $b_1=b_2$ and $a_1<a_2$.
\end{itemize}

The number of basic commutators is given in the following:

\begin{thm}
$($Witt Formula$)$. The number of basic commutators of weight $n$ on $d$
generators is given by the following formula:
$$ \chi_n(d)=\frac {1}{n} \sum_{m|n}^{} \mu (m)d^{n/m},$$
where $\mu (m)$ is the M\"{o}bius function, which is defined to be
   \[ \mu (m)=\left \{ \begin{array}{ll}
      1 & \textrm{if}\ \ m=1, \\ 0 &  \textrm{if} \ \ m=p_1^{\alpha_1}\ldots
p_k^{\alpha_k}\ \ \exists \alpha_i>1, \\ (-1)^s & \textrm{if} \ \
m=p_1\ldots p_s,
\end{array} \right.  \] where the $p_i$, $1\leq i\leq k$, are the distinct primes dividing
$m$.
\end{thm}

\begin{thm}\label{hh}
$($\cite[M. Hall]{m} and \cite[Theorem 11.15(a)]{Hu}$)$
Let $F$ be a free group on $\{x_1, x_2, \ldots, x_d\}.$ Then for all $i$, $1 \leq i \leq n$, \[\dfrac{\gamma_n(F)}{\gamma_{n+i}(F )}\] is a free abelian group freely generated by the basic commutators of weights $n, n + 1, \ldots, n + i - 1$ on the letters $\{x_1, x_2, \ldots, x_d\}.$
\end{thm}
Recall that from \cite{6} if $N $ is  a normal subgroup of $G,$ then we can form the non-abelian tensor product $ (N \otimes G) \otimes G.$ Therefore $ N\otimes_c G $ is defined recursively by $N\otimes_0 G = N$ and $N\otimes_{c+1}G = (N\otimes_{c}G)\otimes G$. Moreover, the exterior product $ N\wedge_c G $ is defined recursively by $N\wedge_0 G = N$ and $N\wedge_{c+1}G = (N\wedge_{c}G)\wedge G$.\\
Here we  give some results concerning  $\mathcal{M}^{(c)}(G)$ which are
used in the proof of the main results. This result are proved in \cite{22} for any variety, but we need them only for the variety of nilpotent groups of class at most $c$.

\begin{thm}\label{1}\cite[Lemma 2.2]{22}
Let $G$ be a finite group, $B\unlhd G$ and $A=G/B$. Then there exists a finite group $L$ with a normal subgroup $M$ such that
\begin{itemize}
\item[\textnormal{(i)}]$\gamma_{c+1} (G) \cap B\cong L/M$,
\item[\textnormal{(ii)}]$M \cong \mathcal{M} ^{(c)} (G)$,
\item[\textnormal{(iii)}]$\mathcal{M} ^{(c)} (A)$ is a homomorphic image of $L$.
\end{itemize}
\end{thm}
%\begin{proof}
%Let $1\rightarrow R \rightarrow F\rightarrow G \rightarrow 1$ be a free presentation of $G$ and suppose $B=S/R$ so that $A=G/B\cong F/S$. Now, set $M =\dfrac{\gamma_{c+1}(F)\cap R}{\gamma_{c+1}(R,F)}$ and
% $L=\dfrac{\gamma_{c+1}(F)\cap S}{\gamma_{c+1}(R,F)}$. Then
% \begin{align*}
%  \gamma_{c+1}(G)\cap B&= \dfrac{\big{(}\gamma_{c+1}(F)R\big{)}\cap S}{R}=\dfrac{\big{(}\gamma_{c+1}(F)\cap S\big{)}~R}{R}\\
%  &\cong \dfrac{\gamma_{c+1}(F)\cap S}{\gamma_{c+1}(F) \cap R} \cong \dfrac{\dfrac{\gamma_{c+1}(F)\cap S}{\gamma_{c+1}(R,F)}}
%  {\dfrac{\gamma_{c+1}(F)\cap R}{\gamma_{c+1}(R,F)}}=\dfrac{L}{M}.
% \end{align*}
%By definition, $\mathcal{M}^{(c)}(G)\cong M$. Now we have
% \begin{align*}
% \mathcal{M}^{(c)}(A)= \dfrac{\gamma_{c+1}(F)\cap S}{\gamma_{c+1}(S,F)}
%  \cong \dfrac{\dfrac{\gamma_{c+1}(F)\cap S}{\gamma_{c+1}(R,F)}}
%  {\dfrac{\gamma_{c+1}(S,F)}{\gamma_{c+1}(R,F)}}\cong \dfrac{L}{\dfrac{\gamma_{c+1}(S,F)}{\gamma_{c+1}(R,F)}}.
% \end{align*}
% Therefore, $\mathcal{M}^{(c)}(A)$ is a homomorphic image  of $L$ under some homomorphism whose kernel is $\dfrac{\gamma_{c+1}(S,F)}{\gamma_{c+1}(R,F)}$.
%\end{proof}
\begin{cor}\label{2}
Let $G$ be a finite group, $B\unlhd G$ and $A=G/B$. Then \[\big{|}\mathcal{M}^{(c)}(A)\big{|} ~\text{divides}~
\dfrac{\big{|}\mathcal{M}^{(c)}(G)\big{|}~\big{|}\gamma_{c+1}(G)\cap B\big{|}}{\big{|}\gamma_{c+1}(B,G)\big{|}}.\]
\end{cor}
\begin{proof}
Let $1\rightarrow R \rightarrow F\rightarrow G \rightarrow 1$ be a free presentation of $G$ and suppose $B=S/R$ so that $A=G/B\cong F/S$. By Theorem \ref{1}, we have
\[
\big{|}\mathcal{M}^{(c)}(A)\big{|}~\Bigg{|}\dfrac{\gamma_{c+1}(S,F)}{\gamma_{c+1}(R,F)}\Bigg{|}=
\big{|}\mathcal{M}^{(c)}(G)\big{|}~\big{|}\gamma_{c+1}(G)\cap B\big{|}.
\]
Since
\[
\gamma_{c+1}(B,G)\cong \Big{[} \dfrac{S}{R}{,}_{c} \dfrac{F}{R}\Big{]}= \dfrac{\gamma_{c+1}(S,F)~R}{R}\cong
\dfrac{\gamma_{c+1}(S,F)}{R\cap \gamma_{c+1}(S,F)},
\]

\[
\Bigg{|}\dfrac{\gamma_{c+1}(S,F)}{\gamma_{c+1}(R,F)}\Bigg{|}=
\Bigg{|}\dfrac{\gamma_{c+1}(S,F)}{R\cap \gamma_{c+1}(S,F)}\Bigg{|}~
\Bigg{|}\dfrac{R\cap \gamma_{c+1}(S,F)}{\gamma_{c+1}(R,F)}\Bigg{|}
\]
\[=\Bigg{|}\gamma_{c+1}(B,G)\Bigg{|}\Bigg{|}\dfrac{R\cap \gamma_{c+1}(S,F)}{\gamma_{c+1}(R,F)}\Bigg{|}.
\]
It follows that
\[
\big{|}\mathcal{M}^{(c)}(A)\big{|}~~ divides ~~\dfrac{\big{|}\mathcal{M}^{(c)}(G)\big{|}~\big{|}\gamma_{c+1}(G)\cap B\big{|}}{\big{|}\gamma_{c+1}(B,G)\big{|}}.
\]
\end{proof}
 For a group $ G,$ the quotient $G/G'$ is denoted by $G^{ab}$.
\begin{lem}\label{3}
Let $G$ be a finite group and $1\rightarrow R\rightarrow F \rightarrow G \rightarrow 1$ be a free presentation of $G$. Let $B$ be a central subgroup of $G$ with $B=S/R$. Put $A=G/B\cong F/S$. Then
$\dfrac{\gamma_{c+1}(S,F)}{\gamma_{c+1}(R,F)~\gamma_{c+1}(S)}$
is a homomorphic image of $B \otimes_{c}\Big{(}\dfrac{G}{B}\Big{)}^{ab}$.
\end{lem}
\begin{proof}
 Define
 \begin{align*}
&\theta :
B \times \underbrace{  \Big{(} \dfrac{G}{B}\Big{)}^{ab} \times \ldots \times  \Big{(} \dfrac{G}{B}\Big{)}^{ab}}
_\text{c-times}\longrightarrow \dfrac{\gamma_{c+1}(S,F)}{\gamma_{c+1}(R,F)~\gamma_{c+1}(S)} \\
&(sR, f_1F' S, \ldots,  f_cF' S)\longmapsto [s,f_1,\ldots, f_c] ~\gamma_{c+1}(R,F)~\gamma_{c+1}(S)
\end{align*}
where $ f_1,\ldots ,f_c\in F, s\in S $. Put $ T=\gamma_{c+1}(R,F)~\gamma_{c+1}(S)$. We claim that $ \theta $ is well-defined.
Let $ f_i= f'_{i}x_iy_i$ where $ f'_i\in F, x_i\in F',y_i\in S$ and $ 1\leq i\leq c$ and also
$ s=s_1r' $ with $ r'\in R,s_1\in S $. It is easy to see that $ [s,f_1,\ldots ,f_c]\equiv [s_1,f'_1,\ldots ,f'_c]  $
mod $ T $. This follows that $ \theta $ is well-defined. Therefore  there exists a unique homomorphism
 \begin{align*}
&\overline{\theta}  :B \otimes \underbrace{ \Big{(} \dfrac{G}{B}\Big{)}^{ab} \otimes \ldots \otimes  \Big{(} \dfrac{G}{B}\Big{)}^{ab}}_\text{c-times}\longrightarrow \dfrac{\gamma_{c+1}(S,F)}{\gamma_{c+1}(R,F)~\gamma_{c+1}(S)} \\
&sR\otimes f_1F' S\otimes \ldots \otimes f_cF'S\longmapsto [s,f_1,\ldots ,f_c] ~\gamma_{c+1}(R,F)~\gamma_{c+1}(S),
\end{align*}
such that $ \mathrm{Im}~ \overline{\theta}= \dfrac{\gamma_{c+1}(S,F )}{\gamma_{c+1}(R,F)~\gamma_{c+1}(S)} $.
\end{proof}
\begin{thm}\label{4}
Let $G$ be a finite group, $B\subseteq Z(G)$ and $A=G/B$. Then
\[
\Big{|} \mathcal{M} ^{(c)} (G)\Big{|} ~\Big{|} \gamma_{c+1}(G) \cap B\Big{|} ~divides ~\Big{|} \mathcal{M} ^{(c)} (A)\Big{|}~\Big{|} \mathcal{M}^{(c)} (B)\Big{|}~\Big{|}B \otimes_{c} \Big{(}\dfrac{G}{B}\Big{)}^{ab}\Big{|}.
\]
\end{thm}
\begin{proof}
Let $1\rightarrow R\rightarrow F \rightarrow G \rightarrow 1$ be a free presentation of $G$, and $B$ be a central subgroup of $G$. Put $B=S/R$ so that $A\cong F/S$. Clearly, $[F,S]\subseteq R$. By Theorem \ref{1} and its proof,
\[
\Big{|} \mathcal{M}^{(c)}(G)\Big{|} ~ \Big{|} \gamma_{c+1}(G) \cap B\Big{|}=\Big{|} \mathcal{M}^{(c)}(A)\Big{|}~
\Big{|}\dfrac{\gamma_{c+1}(S,F)}{\gamma_{c+1}(R,F)} \Big{|}.
\]
Since
\[
\dfrac{\dfrac{\gamma_{c+1}(S,F)}{\gamma_{c+1}(R,F)}}{\dfrac{\gamma_{c+1}(R,F)~\gamma_{c+1}(S)}{\gamma_{c+1}(R,F)}}
\cong \dfrac{\gamma_{c+1}(S,F)}{\gamma_{c+1}(R,F)~\gamma_{c+1}(S)},
\]
we have
\[
\Big{|} \mathcal{M} ^{(c)} (G)\Big{|} ~\Big{|} \gamma_{c+1}(G) \cap B\Big{|}=\Big{|} \mathcal{M} ^{(c)} (A)\Big{|}~\Big{|} \dfrac{\gamma_{c+1}(S,F)}{\gamma_{c+1}(R,F)~\gamma_{c+1}(S)} \Big{|} ~\Big{|} \dfrac{\gamma_{c+1}(R,F)\gamma_{c+1}(S)}{\gamma_{c+1}(R,F)} \Big{|}.
\]
But
\[
\dfrac{\gamma_{c+1}(R,F)\gamma_{c+1}(S)}{\gamma_{c+1}(R,F)}\cong  \dfrac{\gamma_{c+1}(S)}{\gamma_{c+1}(S)\cap \gamma_{c+1}(R,F)}\cong \dfrac{\dfrac{\gamma_{c+1}(S)}{\gamma_{c+1}(R,S)}}
{\dfrac{\gamma_{c+1}(S)\cap \gamma_{c+1}(R,F)}{\gamma_{c+1}(R,S)}}.
\]
Now since $\gamma_{c+1}(S)\subseteq \gamma_{c+1}(S,F)\subseteq [S,F] \subseteq R$,
we have
\[
\dfrac{\gamma_{c+1}(S)}{\gamma_{c+1}(R,S)}=\dfrac{\gamma_{c+1}(S)\cap R}{\gamma_{c+1}(R,S)}=\mathcal{M}^{(c)}(B).
\]
Therefore
\[
\Big{|}\mathcal{M}^{(c)}(G)\Big{|} ~ \Big{|}\gamma_{c+1}(G)\cap B\Big{|}=
\dfrac{\Big{|} \mathcal{M}^{(c)}(A)\Big{|}~\Big{|} \mathcal{M}^{(c)}(B)\Big{|}~\Bigg{|}\dfrac{\gamma_{c+1}(S,F)}{\gamma_{c+1}(S)\gamma_{c+1}(R,F)} \Bigg{|}}
{\Bigg{|}\dfrac{\gamma_{c+1}(S)\cap\gamma_{c+1}(R,F)}{\gamma_{c+1}(R,S)} \Bigg{|}}.
\]
Finally by Lemma \ref{2},
\[
\Big{|} \mathcal{M} ^{(c)} (G)\Big{|} ~\Big{|} \gamma_{c+1}(G)\cap B \Big{|} ~divides ~\Big{|} \mathcal{M}^{(c)} (A)\Big{|}~\Big{|} \mathcal{M} ^{(c)} (B)\Big{|}~\Big{|}B \otimes_c \Big{(}\dfrac{G}{B}\Big{)}^{ab}\Big{|}.
\]
\end{proof}
We just need the results of the next three propositions, so they come without proof.
\begin{prop}\label{5}
\cite[Theorem 2.6]{6}
Let $G$ be a group and $B$ be a central subgroup of $G$. Then the sequence
\[
B \wedge_{c} G\longrightarrow \mathcal{M}^{(c)}(G) \longrightarrow  \mathcal{M}^{(c)}(G/B)\longrightarrow B\cap \gamma_{c+1}(G) \longrightarrow 1
\]
is exact.
\end{prop}
G. Ellis in \cite{7} has generalized the formula of M. R. R. Moghaddam in \cite{24}, for the $c$-nilpotent multiplier of the direct product of two groups. In the next theorem $\Gamma_{c+1} (A, B)$ for abelian groups $A$ and $B$ is defined as follows. Let $A$ and $B$ be $d$-generator and $d'$-generator abelian groups with generating sets $\{a_1,\ldots, a_d\}$ and $\{b_1,\ldots, b_{d'}\}$, respectively. For each basic commutator of weight $c+1$ on $\{a_1,\ldots, a_d,b_1,\cdots, b_{d'}\}$ which is ordered as $a_1<a_2<\cdots<a_d<b_1<b_2<\cdots<b_{d'}$ such as $[x_{i_1},\cdots, x_{i_{c+1}}$, (with any bracketing), we correspond an iterating tensor product $X_1\otimes \cdots \otimes X_{c+1}$ in which $X_{i_j}=A$ if $x_{i_j}\in \{a_1,\ldots, a_d\}$ and $X_{i_j}=B$ if $x_{i_j}\in \{b_1,\ldots, b_{d'}\}$. Now $\Gamma_{c+1} (A, B)$ is the direct sum of all such iterated tensor products involving at least one $A$ and one $B$. (See \cite{7} for more information).
\begin{prop}\label{61}
\cite[Theorem 5]{7}
Let $G$ and $H$ be  finite groups. Then there is an isomorphism
\[
\mathcal{M}^{(c)}(G\times H)\cong \mathcal{M}^{(c)}(G)\oplus \mathcal{M}^{(c)}( H)\oplus \Gamma_{c+1} (G^{ab},H^{ab}).
\]

\end{prop}
 Let  $\mathbb{Z}_ n^{(m)}$ denote the direct product of $m$ copies of $\mathbb{Z}_ n.$
\begin{prop}\label{7}
\cite[Theorem 2.4]{21}
Let $G\cong \mathbb{Z}_{n_1}\oplus\ldots\oplus \mathbb{Z}_{n_k}$, where $n_{i+1}\Big{|} n_i$, $1\leq i \leq k-1$. Then
\[
\mathcal{M}^{(c)}(G)\cong \mathbb{Z}^{(\chi _{c+1}(2))}_{n_2} \oplus \mathbb{Z}^{(\chi _{c+1}(3)-\chi _{c+1}(2))}_{n_3} \oplus\ldots \oplus
\mathbb{Z}^{(\chi _{c+1}(k)-\chi _{c+1}(k-1))}_{n_k}.
\] Where $\chi _{r}(s)$ is the number of all basic commutators of weight $r$ on $s$ letters.
\end{prop}
\begin{cor}\label{8}
Let $G\cong \mathbb{Z}_{p^{m_1}}\oplus\ldots\oplus \mathbb{Z}_{p^{m_k}}$, where $m_1 \geq \ldots \geq m_k$. Then
\[
\mathcal{M}^{(c)}(G)\cong  \bigoplus_{i=2}^{k}\mathbb{Z}^{(\chi _{c+1}(i)-\chi _{c+1}(i-1))}_{p^{m_i}}.
\]

\end{cor}
\section{$c$-nilpotent multipliers of finite $p$-groups}
In this section, we intend to obtain an upper bound for the order of the $c$-nilpotent multiplier of a finite $p$-group. While we're working on the $p$-groups, the extra-special $p$-groups can not be overlooked. At first, we compute the $c$-nilpotent multipliers of these groups. The following theorem of Beyl et al. to see the analogies between the Schur multiplier and the $c$-nilpotent multiplier of extra-special $p$-groups (see Theorem \ref{14}).
\begin{thm}\label{9}\textnormal{(Beyl and Tappe 1982) \cite[Theorem 3.3.6]{18}}
  Let $G$ be an extra-special $p$-group of order $p^{2n+1}$.
\begin{itemize}
\item[\textnormal{(i)}] If $n>1$, then $G$ is unicentral and $\mathcal{M}(G)$ is an elementary abelian $p$-group of order $p^{2n^2 -n-1}$.
\item[\textnormal{(ii)}] Suppose that $G$  of order $p^3$ and $p$ is odd. Then $\mathcal{M}(G)\cong\mathbb{Z}_p \times \mathbb{Z}_p $ if $G$ is of exponent $p$ and $\mathcal{M}(G)=0$ if $G$ is of exponent $p^2$.
\item[\textnormal{(iii)}] The quaternion group of order 8, $ Q_8, $ has trivial multiplier, whereas the multiplier of the dihedral group of order $8$, $ D_8 $, is of order $2$.
\end{itemize}
\end{thm}
The concept of epicenter $Z^*(G)$ is defined by Beyl et al. in \cite{3}. It gives a
criterion for detecting capable groups. In fact $G$ is capable if and only if $Z^*(G) = 1.$
Ellis in \cite{8} defined the exterior center $Z^{\wedge}(G)$ of $G$ the set of all elements $g$ of $G$ for which
$g \wedge h = 1_{G\wedge G}$ for all $h \in  G$ and he showed $Z^*(G) = Z^{\wedge}(G).$\newline
Using the next lemma, the $c$-nilpotent multiplier of non-capable extra-special $p$-groups can be easily computed.

\begin{lem}\label{41}$($See \cite[Theorem 4.2]{3}$)$
Let $N$ be a central subgroup of a group $G$. Then $ N\subseteq Z^{\wedge}(G)$ if and only if $\mathcal{M}(G)\rightarrow \mathcal{M}(G/N)$ is a monomorphism.
\end{lem}

\begin{thm}\label{10}
Let $G$ be a non-capable extra-special $p$-group of order $ p^n$ and $c\geq 2$. Then
\[
\mathcal{M}^{(c)}(G)\cong \mathcal{M}^{(c)}(G/G').
\]
\end{thm}
\begin{proof}
Since $G$ is a non-capable extra-special $p$-group, $Z^{*}(G)=G'$. Now by virtue of Proposition \ref{5}, the sequence
\[
G' \wedge_{c} G \stackrel{\tau}{\longrightarrow} \mathcal{M}^{(c)}(G)\longrightarrow \mathcal{M}^{(c)}(G^{ab})\longrightarrow
 G'\cap \gamma_{c+1}(G) \longrightarrow 1
 \]
 is exact. The rest of proof is obtained by the fact that $G$ is nilpotent of class $2$ and Im $\tau=0$. Therefore
 \[\mathcal{M}^{(c)}(G)\cong \mathcal{M}^{(c)}(G/G').\]
 \end{proof}
\begin{cor}
Let $G$ be a non-capable extra-special $p$-group of order $p^{2n+1}$ and $c\geq 2$. Then $$\mathcal{M}^{(c)}(G)\cong \mathbb{Z}_p^{\big(\chi _{c+1}(2n)\big)}.$$
\end{cor}
\begin{proof}
By  Corollary \ref{8}  and Theorem \ref{10}, we have $\mathcal{M}^{(c)}(G)\cong \mathbb{Z}_p^{\big(\chi_{c+1}(2n)\big)}.$
\end{proof}

There are only two capable extra-special $p$-groups, $D_8$ and $E_1,$ in which $E_1$ is the extra-special $p$-group of order $p^3$ and exponent $p$ for odd $p,$ in \cite{3} . In the following we compute $\mathcal{M}^{(c)}(E_1)$ which has the following presentation.
\[
E_1=\langle x, y|x^p=y^p=[y,x]^p=[y,x,y]=[y,x,x]=1\rangle.
\]
We have $\mathcal{M}^{(c)}(E_1)=\dfrac{R\cap \gamma_{c+1}(F)}{[R,_{c}F]}$  in which $F$ is the free group on the set $\{ x,y\}$ and
$R=\langle x^p, y^p,[y,x]^p,[y,x,y],[y,x,x]\rangle ^F.$\\
Since $E_1$ is nilpotent of class 2, we have $\gamma_{c+1}(F)\subseteq \gamma_3(F)\subseteq R$. Therefore for every $c\geq 2$,
\[
\mathcal{M}^{(c)}(E_1)=\dfrac{\gamma_{c+1}(F)}{[R,_{c}F]}\cong \dfrac{\dfrac{\gamma_{c+1}(F)}{\gamma_{c+3}(F)}}{\dfrac{[R,_{c}F]}{\gamma_{c+3}(F)}}.
\]
We know that $\dfrac{\gamma_{c+1}(F)}{\gamma_{c+3}(F)}$ is the free abelian group with the basis of all basic commutators of weights $c+1$ and $c+2$ on $\{x,y\}$, by Theorem \ref{hh}.
\begin{thm}\label{11}
With the above notations and assumptions we have
\[
[R,_{c}F]\equiv \Big{(} \gamma_{c+1}(F)\Big{)}^p \pmod{\gamma_{c+3}(F)}, ~~for ~all~c\geq 2.
\]
\end{thm}
\begin{proof}
Using induction on $c$, the result is true for $c=2$ (\cite[Theorem 2.6]{29A}). Assume that the result holds for $c\geq 2$. By hypothesis, we have
\[
[R,_{c+1} F]\equiv \Bigg{[}\Big{(}\gamma_{c+1}(F)\Big{)}^p,F\Bigg{]} \pmod{\gamma_{c+4}(F)}.
\]
It follows that
\[
[R,_{c+1} F]\equiv \Big{(}\gamma_{c+2}(F)\Big{)}^p \pmod{\gamma_{c+4}(F)}.
\]
\end{proof}
\begin{thm}\label{12}
With that above notations and assumptions we have
\[
\mathcal{M}^{(c)}(E_1)\cong \mathbb{Z}_{p}^{\big(\chi _{c+1}(2)+\chi_{c+2}(2)\big)},~~\text{for} ~~c\geq 2.
\]
\end{thm}
\begin{proof}
By Theorem \ref{hh}, we know that $\dfrac{\gamma_{c+1}(F)}{\gamma_{c+3}(F)}$ is the free abelian group with the basis of all basic commutators of wight $c+1$ and $c+2$ on $\{ x,y\}$. By Theorem \ref{11}, we have
\[
\dfrac{[R,_{c}F]}{\gamma _{c+3}(F)}=\Bigg{(} \dfrac{\gamma_{c+1}(F)}{\gamma_{c+3}(F)}\Bigg{)}^p,\]
so the result holds.
\end{proof}
\begin{thm}\cite[Theorem C]{111}\label{13}
Let $G$ be a  group and $1\rightarrow R\rightarrow F \rightarrow G \rightarrow 1$ be a free presentation of $G$. Put $\gamma_{c+1}^{*}(G)=\dfrac{\gamma_{c+1}(F)}{\gamma_{c+1}(R,F)}$, then
\begin{itemize}
\item[\textnormal{(i)}] $0\longrightarrow \mathcal{M}^{(c)}(G) \longrightarrow \gamma_{c+1}^{*}(G) \longrightarrow \gamma_{c+1}(G)\longrightarrow 0$ is exact.
\item[\textnormal{(ii)}] For $n\geq 2$, we have
\begin{align*}
\gamma_{c+1} ^{*}(D_n)\cong
\left \lbrace
\begin{array}{ll}
\mathbb{Z}_n &n ~ is ~odd\\
\mathbb{Z}_n \times \mathbb{Z}^{(\chi _{c+1}(2)-1)}_{2} & n ~ is ~even
\end{array}\right.,
\end{align*}
where $ D_n $ is dihedral group of order $2n.$
\end{itemize}
\end{thm}
It is known that $\mathcal{M}^{(c)}(D_8)\cong \mathbb{Z}_4 \times \mathbb{Z}^{(\chi_{c+1}(2)-1)}_{2}$. Also we know that $E_2,$ the extra-special $p$-group of order $p^3$ and exponent $p^2 $ for odd $p,$ and $Q_8$ are not capable groups, so we can summarize the explicit structure of the $c$-nilpotent multipliers of all extra-special $p$-groups as follows, for $c\geq 2$.
\begin{thm}\label{14}
Let $G$ be an extra-special $p$-group of order $p^{2n+1}$ and $c\geq 2$.
\begin{itemize}
\item[\textnormal{(i)}] If $n>1$, then $\mathcal{M}^{(c)}(G)$ is an elementary abelian $p$-group of order $ p^{\chi_{c+1}(2n)}$.
\item[\textnormal{(ii)}] Suppose that $G$ of order $p^3$ and $p$ is odd. Then $\mathcal{M}^{(c)}(G)\cong \mathbb{Z}^{(\chi_{c+1}(2)+\chi_{c+2}(2))}_{p}$ if $G$ is of exponent $p$ and $\mathcal{M}^{(c)}(G)\cong\mathbb{Z}^{(\chi_{c+1}(2))}_{p}$ if $G$ is of exponent $p^2$.
\item[\textnormal{(iii)}] The quaternion group of order 8, $ Q_8 $, has $\mathbb{Z}^{(\chi_{c+1}(2))}_{2}  $ as $c$-nilpotent multiplier and the $c$-nilpotent multiplier of the dihedral group of order 8, $ D_8$,  is\ \  $\mathbb{Z}_4 \times \mathbb{Z}^{(\chi_{c+1}(2)-1)}_{2}$.
\end{itemize}
\end{thm}
In the proof of some of next theorems, we need to work with the $c$-nilpotent multiplier of finite abelian $p$-groups. For abelian $p$-groups we have.
\begin{thm}\label{15}
Let $G\cong \mathbb{Z}_{p^{m_1}} \oplus\ldots \oplus\mathbb{Z}_{p^{m_k}}$, where $m_1\geq m_2\geq \ldots \geq m_k$ and $ |G|=p^n $. Then
\begin{itemize}
\item[\textnormal{(i)}]  $\big{|} \mathcal{M}^{(c)}(G)\big{|} = p^{\chi_{c+1}(n)}$ if and only if $m_i=1$ for all $i$.
\item[\textnormal{(ii)}] $\big{|} \mathcal{M}^{(c)}(G)\big{|} \leq p^{\chi_{c+1}(n-1)}$ if and only if $m_1\geq 2$.
\end{itemize}
\end{thm}
\begin{proof}
Using Corollary $ \ref{8} $, we have $ |\mathcal{M}^{(c)}(G)|=p^{\sum_{i=2}^{k}(\chi_{c+1}(i)-\chi_{c+1}(i-1))m_i} $. Define
\[S(k:m_1,\ldots,m_k)=\sum_{i=2}^{k}(\chi_{c+1}(i)-\chi_{c+1}(i-1))m_i.\]
It is easy to see that $ S(n:1,\ldots,1) =\chi_{c+1}(n)$ and $ S(n:2,1,\ldots,1) =\chi_{c+1}(n-1)$. Assume that for some $ j $,
we have $ m_j>1 $ an straightforward computation shows that
\[S(k:m_1,\ldots,m_k)-S(k+1:m_1,\ldots,m_j-1,\ldots,m_k,1)=\]
\[(\chi_{c+1}(j-1)-\chi_{c+1}(j))+(\chi_{c+1}(k)-\chi_{c+1}(k+1))\]
which is negative. Hence the maximum  value of $S(k:m_1,\ldots,m_k)  $ is $ S(n:1,\ldots,1)  $ and the next largest value
of $S(k:m_1,\ldots,m_k)  $ is $ S(n-1:2,1,\ldots,1).$
\end{proof}
Here we recall the notion of central product of two groups.
\begin{defn}
The group $G$ is a central product of $A$ and $B$, if $G = AB$, where $A$ and $B$ are normal subgroups of $G$ and $A \subseteq C_G(B)$.
\end{defn}

We denote the central product of two groups $A$ and $B$ by $A\cdot B$.\\
Recall that a group $G$ is said to be minimal non-abelian if it is non-abelian but all its proper
subgroups are abelian. \\
%Here we recall the notion of central product of two groups.
%\begin{defn}
%The group $G$ is a central product of $A$ and $B$, if $G = AB$, where $A$ and $B$ are normal subgroups of $G$ and $A \subseteq C_G(B)$.
%\end{defn}
%We denote the central product of two groups $A$ and $B$ by $A\cdot B$.
The number minimal generators of group of $ G $ and the Frattini subgroup of $ G $ are denoted by $ d(G)$ and $\Phi(G),$ respectively.\\
%The number minimal generators of a group $ G $ is denoted by $ d(G).$
This lemma gives information on minimal non-abelian $p$-groups.

\begin{lem}\cite [ Lemma 2.2]{zh}\label{401}
Assume that $G$ is a finite nonabelian $p$-group. Then the
following conditions are equivalent:
\begin{itemize}
\item[$(1)$] $G$ is minimal non-abelian;
\item[$(2)$] $d(G) = 2$ and $|G'| = p$;
\item[$(3)$] $d(G) = 2$ and $\Phi(G) = Z(G)$.
\end{itemize}
\end{lem}
For $p$-groups with small derived subgroup we have:
\begin{lem}\label{151}
\cite[Lemma 4.2]{ber}
Let $G$ be a $p$-group with $|G'|= p$. Then $G=(A_1\cdot A_2\cdot \ldots\cdot A_s)Z(G)$,
the central product, where $A_1,\ldots , A_s$ are minimal non-abelian groups. Furthermore $G/Z(G)$ is an elementary
abelian group of even rank.
\end{lem}
The order of the $c$-nilpotent multiplier of a finite $p$-group depends somehow on the order of its derived subgroup.
\begin{lem}\label{16}
Let $G$ be a non-abelian finite $p$-group of order $p^n$ whose derived subgroup is of order $p$ and $ c\geq 2 $. Then
$\Big{|} \mathcal{M}^{(c)}(G)\Big{|} \leq p^{\chi_{c+1}(n-1)+\chi_{c+2}(2)}$, and the equality holds if and only if $G\cong E_1\times \mathbb{Z}_{p}^{(n-3)}$.
\end{lem}
\begin{proof}
By Lemma \ref{151}, $G=(A_1\cdot A_2\cdot\ldots\cdot A_s)Z(G)$ where $A_i$'s are minimal non-abelian groups. First suppose that $Z(G)$ is not elementary abelian. Then by Theorem \ref{4}, we have
\[
\Big{|} \mathcal{M}^{(c)}(G)\Big{|}\Big{|}\gamma_{c+1}(G)\cap Z(G)\Big{|}\leq  \Big{|}\mathcal{M}^{(c)}(G/ Z(G))\Big{|} ~
\Big{|} \mathcal{M}^{(c)}(Z(G))\Big{|}~\Big{|}  Z(G)\otimes_c G/ Z(G)\Big{|}.
\]
Using Proposition \ref{61} and Theorem \ref{15}$ (ii)$, since $  Z(G)\otimes_c G/ Z(G)  $ is the direct summand of $\Gamma_{c+1}(Z(G), G/ Z(G)),$ we have \[\Big{|}  Z(G)\otimes_c G/ Z(G)\Big{|}\leq \Big{|}\Gamma_{c+1}(Z(G), G/ Z(G)) \Big{|}.\] Therefore
\[\Big{|} \mathcal{M}^{(c)}(G)\Big{|}\leq \Big{|} \mathcal{M}^{(c)}(Z(G)\times G/ Z(G))\Big{|}\leq p^{\chi_{c+1}(n-1)}\leq p^{\chi_{c+1}(n-1)+\chi_{c+2}(2)}.\]
Now assume that $Z(G)$ is elementary abelian.
 If $|Z(G)|=p,$ then $G$ is extra-special due to Theorem \ref{14}.  We may assume that $|Z(G)| \geq p^2.$ If $ G/G' $  is not elementary abelian, then by a similar way in the previous case and using Theorem \ref{4} and  Proposition \ref{61}, we have
\[
\Big{|} \mathcal{M}^{(c)}(G)\Big{|}\Big{|}\gamma_{c+1}(G)\cap G'\Big{|}\leq  \Big{|}\mathcal{M}^{(c)}(G/ G')\Big{|} ~
\Big{|} \mathcal{M}^{(c)}(G')\Big{|}~\Big{|}  G'\otimes_c G/ G'\Big{|}.
\]
\[\leq \Big{|}\mathcal{M}^{(c)}(G'\times G/G')\Big{|}.\]
Thus Theorem \ref{15}$ (ii),$ implies $\Big{|} \mathcal{M}^{(c)}(G)\Big{|}\leq  p^{\chi_{c+1}(n-1)}\leq p^{\chi_{c+1}(n-1)+\chi_{c+2}(2)}.$
 \newline
 Now assume that $G/G'$ is elementary abelian. It is shown that \cite[Lemma 2.2]{25}, $G$ is a central product of an extra-special $p$-group $H$ of order $p^{2k+1}$ and $Z(G)$ of order $p^{n-2k}$. %Again we may assume that $|Z(G)| \geq p^2$ due to Theorem \ref{14}.
%  We may assume that $|Z(G)| \geq p^2$ due to Theorem \ref{14}.
% If $Z(G)$ is not elementary abelian, then by Theorem \ref{4}, we have
%\[
%\Big{|} \mathcal{M}^{(c)}(G)\Big{|}\leq  \Big{|}\mathcal{M}^{(c)}(G/ Z(G))\Big{|} ~
%\Big{|} \mathcal{M}^{(c)}(Z(G))\Big{|}~\Big{|}  Z(G)\otimes_c G/ Z(G)\Big{|}.
%\]
%By Proposition \ref{61} and Theorem \ref{15}$ (ii)$, therefore
%\[\Big{|} \mathcal{M}^{(c)}(G)\Big{|}\leq \Big{|} \mathcal{M}^{(c)}(Z(G)\times G/ Z(G))\Big{|}\leq p^{\chi_{c+1}(n-1)}\leq p^{\chi_{c+1}(n-1)+\chi_{c+2}(2)}.\]
Still in the case that $Z(G)$ is elementary abelian, suppose that $T$ be a complement of $G'$ in $Z(G)$, we have $Z(G)=G'\times T$ and so $G=H\times T$.
 By Proposition \ref{61}, we have
\[
\Big{|} \mathcal{M}^{(c)}(G)\Big{|}=  \Big{|}\mathcal{M}^{(c)}(H\times T)\Big{|}=\Big{|} \mathcal{M}^{(c)}(H)\Big{|} ~
\Big{|} \mathcal{M}^{(c)}(T)\Big{|}~\Big{|} \Gamma_{c+1}(H^{ab}, T)\Big{|}.\]
If $ H $ is not capable, then Theorems \ref{10} and \ref{14} imply
\[\Big{|} \mathcal{M}^{(c)}(H/H')\Big{|}=\Big{|} \mathcal{M}^{(c)}(H)\Big{|}=
p^{\chi_{c+1}(2k)}
.\]
By Theorem \ref{15}$ (ii) $, we have
%$\Big{|} M^{(c)}(Z(G))\Big{|}\leq p^{\chi_{c+1}(n-2k-1)}$. By Proposition \ref{61}, we have
%\[
%\Big{|}\Gamma_{c+1}(H^{ab}, Z(G))\Big{|}=p^{\chi_{c+1}(n-1)-\chi_{c+1}(2k)-\chi_{c+1}(n-2k-1)},\\
%\]
%and by Proposition \ref{5}, we have
\[\Big{|} \mathcal{M}^{(c)}(G)\Big{|}=\Big{|} \mathcal{M}^{(c)}(T\times H^{ab})\Big{|}= p^{\chi_{c+1}(n-1)}.\]
%Thus
%\begin{align*}
%\Big{|}\mathcal{M}^{(c)}(G)\Big{|}\leq p \Big{|} \mathcal{M}^{(c)}(Z(G)\times H^{ab})\Big{|}\leq p^{\chi_{c+1}(n-1)+1}
%<p^{\chi_{c+1}(n-1)+\chi_{c+2}(2)}.
%\end{align*}
Now assume that $ H=E_1 $ or $ H=D_8$.
By Proposition \ref{61} and Theorem \ref{15} $ (ii) $, we have
\[
\Big{|} \mathcal{M}^{(c)}(T)\Big{|}=p^{\chi_{c+1}(n-3)},~~ \Bigg{|}\Gamma_{c+1}(H^{ab},T)\Bigg{|}
=p^{\chi_{c+1}(n-1)-\chi_{c+1}(n-3)-\chi_{c+1}(2)}.
\]
Thus
\begin{align*}
\Big{|} \mathcal{M}^{(c)}(G)\Big{|}=\left\{
\begin{array}{ll}
2^{\chi_{c+1}(n-1)+1}&\! \! \! \! \! \! \! \! \! \! \! \! \! \! \! \! \! \! \! \! \! \! \! \! \qquad \qquad if~p=2,\\
p^{\chi_{c+1}(n-1)+\chi_{c+2}(2)}&~~
 ~~\quad\qquad\! \quad\! \! \! \! \! \! \! \! \! \! \! \! \! \! \! \! \! \! \! \! \! \! \!  if~p>2.\\
\end{array}\right .
\end{align*}
Now it is easy to see that $\Big{|} \mathcal{M}^{(c)}(G)\Big{|}=p^{\chi_{c+1}(n-1)+\chi_{c+2}(2)}$ if and only if  $G\cong E_1 \times \mathbb{Z}_{p}^{(n-3)}$.
\end{proof}
Now we are in a position to summarize the results in the following theorem.
\begin{thm}\label{17}
Let $ G $ be a $p$-group of order $ p^n $ with $ |G'|=p^m (m \geq 1)$ and $ c\geq 2.$ Then
\[\Big{|}\mathcal{M}^{(c)} (G)\Big{|}\leq p^{\chi_{c+1}(n-m)+\chi_{c+2}(2)+(m-1)(n-m)^c}.\]
In particular, if $ |G'|=p $, then we have $ \Big{|}\mathcal{M}^{(c)}(G)\Big{|}\leq p^{\chi_{c+1}(n-1)+\chi_{c+2}(2) }$ and the equality holds in last one if  and only if $ G\cong E_1\times  \mathbb{Z}_{p}^{(n-3)}$.
\end{thm}
\begin{proof}
Let $ G $ be an arbitrary non-abelian $ p$-group of order $ p^n.$ We proceed by induction on $ m $. The case $ m=1 $
follows from Lemma \ref{16}. \newline Therefore, we may assume that $ m\geq 2 $. Let $ B $ a central subgroup of order $ p $ in $ G' $, by Theorem \ref{4}, we have
\[\Big{|}\mathcal{M}^{(c)}(G)\Big{|}\leq \Big{|}\mathcal{M}^{(c)}(G)\Big{|}\Big{|}\gamma_{c+1}(G)\cap B\Big{|}\leq \Big{|}\mathcal{M}^{(c)}(G/B)\Big{|}~\Big{|} B \otimes_c G^{ab}\Big{|}.\]
 Using induction hypothesis
\[ \big{|} \mathcal{M}^{(c)}(G/B) \big{|} \leq  p^{\chi_{c+1}(n-m)+\chi_{c+2}(2)+(m-2)(n-m)^c },\]
and so
\[ \big{|} \mathcal{M}^{(c)}(G) \big{|} \leq  p^{\chi_{c+1}(n-m)+\chi_{c+2}(2)+(m-2)(n-m)^c+ (n-m)^c}\]
\[=p^{\chi_{c+1}(n-m)+\chi_{c+2}(2)+(m-1)(n-m)^c },\]
which completes the proof.
\end{proof}
\section{$c$-Capability of extra-special $p$-groups}
\paragraph{}
The first study of the capability of extra-special $p$-groups was by Beyl et al. in \cite{3}. They showed that in extra-special $p$-groups, only $E_1$ and $D_8$ are capable. Recently the last two authors, proved that for extra-special $p$-groups, the notions ``capable'' and ``2-capable'' are equivalent. Here we generalize it and show that for these groups, ``capability'' and ``$c$-capability'' are equivalent. Recall that if $G$ is a $c$-capable group, that is $G\cong E/Z_c(E)$ for some group $E$; then $G\cong \frac{E}{Z_{c-1}(E)}/Z(\frac{E}{Z_{c-1}(E)})$ which shows $G$ is capable too. So we just have to prove the $c$-capability of $D_8$ and $E_1$. To do this we need the following proposition which can be found in \cite[ Lemma 2.1(iv) and Proposition 1.2]{6}. For the statement of this fact we need terminology from \cite{6} as below.

Let $F/R$ be a free presentation for $G$ and $\pi :F/[R,\ _cF]\longrightarrow G$ be the canonical surjection. The $c$-central subgroup $Z_{c}^{*}(G)$ of $G$ is the image in $G$ of the section term of the upper central series of $F/[R,\ _cF]$. More precisely it is equal to $\pi (Z_c(F/[R,\ _cF]))$.
\begin{prop}\label{171}
\cite[Proposition 12]{111}
\begin{itemize}
\item[$(i)$] A group $G$ is $c$-capable if and only if $Z^{*}_{c}(G)$ is trivial;
\item[$(ii)$]If $N$ is normal subgroup of $G$ contained in $Z^{*}_{c}(G)$, then the canonical
\[
\mathcal{M}^{(c)}(G)\longrightarrow \mathcal{M}^{(c)}(G/N)
\]
homomorphisms is injection.
\item[$(iii)$] $ Z_{c+1}^*(G) $ contains $ Z_{c}^*(G) $ for $c \geq 0.$
\end{itemize}
\end{prop}
Proposition \ref{171}  yields the sequences
\[1=Z_{0}^*(G)\subseteq Z_{1}^*(G)\subseteq Z_{2}^*(G)\subseteq Z_{3}^*(G) \subseteq \ldots,\]
we  recall that $ Z_{1}^*(G) $ the epicentre of $G.$\newline
Beyl's results can be extended to $c$-capability as follows.
\begin{thm}\label{18}
An extra-special $p$-group is $c$-capable if and only if $G$ is isomorphic to one of the groups $D_8$ or $E_1$.
\end{thm}
\begin{proof}
%First we show $E_1$ is $c$-capable. We use Proposition 3.1 (ii) and show that for arbitrary normal subgroup $N$ of $E_1$, $M^{(c)}(E_1)\longrightarrow M^{(c)}(E_1/N)$ is not injective. For each normal subgroup $N$ of $E_1$, $E_1/N$ is an abelian $p$-group of order at most $p^2$ so $|M^{(c)}(E_1/N)|\leq p^\chi_{c+1}(2)$, but $|M^{(c)}(E_1)|=p^{\chi_{c+1}(2)+\chi_{c+2}(2)}$. For $D_8$ the proof is completely similar.
%
%The above technique can be used for determining the $c$-capability of generalized extra special $p$-groups.
%
%
Let $G$ be an extra-special $p$-group of order $p^3$ and exponent $p$. We will show there is no nontrivial normal subgroup of $G$ for which the natural homomorphism $\mathcal{M}^{(c)}(G)\longrightarrow \mathcal{M}^{(c)}(G/N)$ is injective. Let $N$ be a nontrivial normal subgroup of $G$, so $G/N$ is an abelian $p$-group of order at most $p^2$ and hence, Theorem \ref{15} shows $\big{|} \mathcal{M}^{(c)}(G/N)\big{|} \leq p^{\chi_{c+1}(2)}$. Since $\big{|} \mathcal{M}^{(c)}(G)\big{|}=p^{\chi_{c+1}(2)+\chi_{c+2}(2)}$ by Theorem \ref{17},
$\mathcal{M}^{(c)}(G)\longrightarrow \mathcal{M}^{(c)}(G/N)$ fails to be injective and the result holds. For the case $G=D_8,$ we have $D_{2^{c+1}}/Z_c(D_{2^{c+1}})\cong D_8.$ The proof is complete.
\end{proof}

\section{$c$-Capability and  $c$-nilpotent multiplier of generalized extra-special $p$-groups}
Jafari et al. in \cite{151}  studied the Schur multiplier  of generalized extra-special $p$-groups. We  obtain the $c$-nilpotent multiplier of a generalized extra-special $p$-groups.
Niroomand and Parvizi in \cite{28} first studied the capability of generalized extra-special $p$-groups. They showed that in generalized extra-special $p$-groups, only  $G=E_1\times \mathbb{Z}_p^{(n-3)}$  and $ G=D_8\times \mathbb{Z}_2^{n-3}$ are capable.  Here we generalize the results and show that for these groups, ``capability'' and ``$c$-capability'' are equivalent. Because non-capable groups are not $c$-capable, we just have to prove the $c$-capability of  $G=E_1\times \mathbb{Z}_p^{(n-3)}$  and $ G=D_8\times \mathbb{Z}_2^{(n-3)}.$

\begin{defn}\label{19}
\cite[Definition 3.1]{29}
Let $ G $ be a finite $p$-group. Then $ G $ is called generalized extra-special $p$-group if $ \Phi(G)=G'\cong \mathbb{Z}_{p} $.
\end{defn}
The structure of generalized extra-special $p$-groups is as follows. Here by $E_{p^{2m+1}}$ we mean the extra-special $p$-group of order $p^{2m+1}$.
\begin{lem}\label{20}
\cite[Lemma 3.2]{29}
Let $ G $ be a generalized extra-special $p$-group. Then if  $ Z(G)\cong \Phi(G) \times A$  then $ G\cong E_{p^{2m+1}} \times A$, and  if  $ Z(G)\cong \mathbb{Z}_{p^2} \times A$  then $ G\cong (E_{p^{2m+1}}\cdot\mathbb{Z}_{p^2} )\times A$, in which $A$ is an elementary abelian $p$-group.
\end{lem}

%\begin{cor}
%et $ G $ be a generalized extra-special $p$-group. Then if  $ Z(G)\cong \Phi(G) \times A$  then $ G\cong E_{p^{2m+1}} \times \mathbb{Z}_{p^{n-2m-1}}$, and  if  $ Z(G)\cong \mathbb{Z}_{p^2} \times A$  then $ G\cong (E_{p^{2m+1}}*\mathbb{Z}_{p^2} )\times \mathbb{Z}_{p^{n-2m-2}} $, in which $A$ is an elementary abelian $p$-group.
%\end{cor}
The capability of the generalized extra-special $p$-groups is determined in the following.
\begin{prop}\cite[Theorem 3.4]{28}\label{21}
Let $ G $ be a generalized extra-special $p$-group. $ G $ is capable if and only if $ G=E_1\times \mathbb{Z}_p^{(n-3)}$  or $ G=D_8\times \mathbb{Z}_2^{(n-3)}$.
\end{prop}

\begin{prop}\label{22}
\cite[Lemma 2.9]{151}
Let $G$ be the generalized extra-special $p$-group and isomorphic to $E_{p^{2m+1}}\cdot B$ such that $ B\cong \mathbb{Z}_{p^2}$. Then $ \big{|}\mathcal{M}(G)\big{|}= p^{2m^2+m-1} $.
\end{prop}
The following proposition will be used in the next investigation.
\begin{prop}\label{221}
Let $ G=E_{p^{2m+1}}\cdot B$ such that $ B\cong \mathbb{Z}_{p^2}.$ Then $\big{|}\mathcal{M}^{c}(G)\big{|}=p^{\chi_{c+1}(2m)} $, for all $ c\geq 2 $.
\end{prop}
\begin{proof}
By Proposition \ref{21}, since $G$ is a non-capable generalized extra-special $p$-group, $Z^{*}(G)=G'$. Now by Proposition \ref{171}, the homomorphism $ \mathcal{M}^{(c)}(G)\longrightarrow \mathcal{M}^{(c)}(G^{ab})$ is an isomorphism. Therefore $ \Big{|} \mathcal{M}^{(c)}(G)\Big{|}=\Big{|} \mathcal{M}^{(c)}(G^{ab})\Big{|}.$ Now the result holds by Theorem \ref{15}.
\end{proof}
The following theorem determines the $c$-capability of the generalized extra-special $p$-groups.
\begin{thm}\label{23}
Let $ G $ be a generalized extra-special $p$-group and $ |G|=p^n $. $ G $ is $c$-capable if and only if $ G=E_1\times  \mathbb{Z}_p^{(n-3)}$  or $ G=D_8\times \mathbb{Z}_2^{(n-3)}$.
\end{thm}
\begin{proof}
Let $G=E_1\times \mathbb{Z}_p^{(n-3)}$ be a generalized extra-special $p$-group. We will show there is no nontrivial normal subgroup of $G$ for which the natural homomorphism $\mathcal{M}^{(c)}(G)\longrightarrow  \mathcal{M}^{(c)}(G/N)$ is injective. Let $N$ be a nontrivial normal subgroup of $G$, so $G/N$ is an abelian $p$-group of order at most $p^{n-1}$ and hence Theorem \ref{15} $(i)$ shows $\big{|} \mathcal{M}^{(c)}(G/N)\big{|} \leq p^{\chi_{c+1}(n-1)}$. Since $\big{|} \mathcal{M}^{(c)}(G)\big{|}=p^{\chi_{c+1}(n-1)+\chi_{c+2}(2)}$ by Theorem \ref{17},
$\mathcal{M}^{(c)}(G)\longrightarrow \mathcal{M}^{(c)}(G/N)$ fails to be injective and the result holds.  Now Proposition \ref{171} shows $G=E_1\times \mathbb{Z}_p^{(n-3)}$ is $c$-capable. The proof for the case $G=D_8\times \mathbb{Z}_{2}^{(n-3)}$ is completely similar expect that $\mathcal{M}^{(c)}(D_8)\cong \mathbb{Z}_{4}\times \mathbb{Z}_{2}^{(\chi_{c+1}(2)-1)} $.
\end{proof}
\begin{cor}\label{24}
Let $ G $ be a generalized extra-special $p$-group. Then $ G $ is $c$-capable if and only if $  G$ is capable.
\end{cor}

\end{document}